\documentclass[12pt,  reqno]{amsart}
\usepackage{amsmath}
\usepackage{amsmath,  amsfonts,  amssymb}
\usepackage{eucal}
\usepackage{mathrsfs}
\usepackage{latexsym}
\usepackage{cite,  color}

\newcommand{\rn}{\mathbb R^n}
\newcommand{\sn}{\mathbb S^{n-1}}

\newcommand{\psum}{{+_{\negthinspace\kern-2pt p}}\, }
\newcommand{\qsum}[1]{{+_{\negthinspace\kern-2pt #1}}\, }
\newcommand{\dpsum}{{\tilde+_{\negthinspace\kern-1pt p}}\, }
\newcommand{\dqsum}[1]{{\tilde+_{\negthinspace\kern-1pt #1}}\, }
\newcommand{\lsub}[1]{\hskip -1.5pt\lower.5ex\hbox{$_{#1}$}}

\begin{document}
	
	\title[Existence of Solutions with small volume to $L_p$-Gaussian Minkowski problem]{Existence of Solutions with small volume to $L_p$-Gaussian Minkowski problem}

	\author[S. Tang]{Shengyu Tang}
	\address{Institute of Mathematics, 
		Hunan University,  Changsha,  410082,  China}
	\email{
	tsy@hnu.edu.cn}
	

	
	\begin{abstract}
		In this paper, we derive the existence of solutions with small volume to the $L_p$-Gaussian Minkowski problem for $1\leq p<n$, which implies that there are at least two solutions for the $L_p$-Gaussian Minkowski problem. 
	\end{abstract}
	\maketitle
	
	\numberwithin{equation}{section}
	
	\newtheorem{theo}{Theorem}[section]
	\newtheorem{coro}[theo]{Corollary}
	\newtheorem{lemm}[theo]{Lemma }
	\newtheorem{prop}[theo]{Proposition}
	\newtheorem{conj}[theo]{Conjecture}
	\newtheorem{exam}[theo]{Example}
	\newtheorem{prob}[theo]{Problem}
	
	\theoremstyle{definition}
	\newtheorem*{defi}{Definition}
	
	\section{Introduction}
	
	 Brunn-Minkowski theory can be viewed as the result of merging two basic concepts  for point sets in Euclidean space: geometric invariant and vector addition. In the classical Brunn-Minkowski theory, the volume of a convex body serves as the geometric invariant, and the addition of two convex bodies is known as Minkowski addition: for two convex bodies $K,L\in \rn$, $K+L:=\{k+l: k\in K,l\in L\}$. In 1938, Aleksandrov provided the variational formula in \cite{aleksandrov1938theory},
	\begin{equation}\label{variational formula}
		\lim_{t\rightarrow0}\frac{V(K+tL)-V(K)}{t}=\int_{\sn}h_L(v)dS_K(v),
	\end{equation}
	where $V(\cdot)$ is the volume of convex body, $\sn$ is the unit sphere in $\rn$, $h_L(v):=\max_{x\in L}\{v\cdot x\}$ for $v\in \sn$ is the support function of the convex body $L$. It provides the classical non-zero finite Borel geometric measure, known as the surface area measure $S_K$. More generally, as the most basic geometric invariants, quermassintegrals include volume and surface area as special cases. When considering quermassintegrals as  functionals defined in a convex body in $\rn$, their derivatives in a certain sense are geometric measures, called area measures and (Federer's) curvature measures. For more details, see Schneider\cite{MR3155183}.
	
	Based on these geometric measures generated from convex bodies, the inverse problem, also known as the Minkowski type problem, was formulated: what conditions are necessary and sufficient for the existence and uniqueness of a convex body such that a given measure is generated by it? Minkowski \cite{MR1511220} first addressed and solved the issue of the prescribed surface area measure for the discrete case, while its general case was subsequently addressed by Aleksandrov \cite{aleksandrov1938theory} and by Fenchel-Jessen \cite{fenchel1938mengenfunktionen} in 1938.
	
	The classical Brunn-Minkowski theory has undergone two major developments, each corresponding to its two fundamental concepts. One question was posed by Firey, who initially replaced the classical Minkowski combination with the $L_p$-Minkowski combination. For two convex bodies $K$ and $L$, denoted by $K+_pL$ the $L_p$-Minkowski combination of $K$ and $L$, if its support function satisfies $h_{K+_pL}(v)=[h_K^p(v)+h_L^p(v)]^{\frac{1}{p}}$, where $h_K$, $h_L$ are the support functions for $K$ and $L$ respectively. Similar to \eqref{variational formula}, the volume variational formula for the $L_p$-Minkowski sum is as follows,
	\[
	\lim_{t\rightarrow0}\frac{V(K+_ptL)-V(K)}{t}=\frac{1}{p}\int_{\sn}h_L^p(v)dS_p(K,v),
	\]
	where $S_p(K,\cdot):=h_K^{1-p}S(K,\cdot)$ is called the $L_p$-surface area measure. Correspondingly, this is related to a Minkowski problem called the  $L_p$-Minkowski problem. It is evident that when $p=1$, $S_p(K,\cdot)=S(K,\cdot)$, and the $L_p$-Minkowski problem reduces to the classical Minkowski problem.  Lutwak solved the existence and uniqueness of convex bodies in symmetric cases for $p>1$ in \cite{MR1231704}. Recent years have seen relevant articles on the existence and uniqueness of solutions to the Minkowski problem when $p<1$. For example, Chen-Huang-Li-Liu in \cite{MR4088419} provided the uniqueness for even conditions when $p_0<p<1$, where $p_0\geq1-\frac{1}{\frac{4n^2}{\pi^2}+\frac{4n}{\pi}}$. Jian-Lu-Wang \cite{MR3366854} showed that the uniqueness can not hold when $p<0$.  When $p=0$, its related Minkowski problem is known as the lograithmic Minkowski problem. B\"{o}r\"{o}czky-LYZ  first solved the even case in \cite{MR3037788}, and B\"{o}r\"{o}czky-Heged\H{u}s-Zhu \cite{MR3509941} established the existence under discrete case. For additional works on the logarithmic Minkowski problem, please see \cite{MR3896091,MR3228445} and so on. When $p=-n$, its related problem is called the centro-affine Minkowski problem and see related research in \cite{MR2254308}.
	More extensive results for the $L_p$-Minkowski problem can be found in \cite{MR2132298,MR3872853,MR3680945,MR3565388}.
	
	The dual Brunn-Minkowski theory is another development in the field of Brunn-Minkowski theory. Lutwak first introduced the dual quermasintegral $\tilde{W}_q(K)$ in \cite{MR380631,MR415505} in 1975. Recently, an exciting variational formula has been discovered by Huang-Lutwak-Yang-Zhang \cite{MR3573332},
	\[
	\frac{d}{dt}\tilde{W}_{n-i}(K_t)|_{t=0}=i\int_{\sn}f(v)d\tilde{C}_i(K,v),  
	\]
	where $K_t=\{x\in\rn:\ x\cdot v\leq h_K(v)+tf(v)\ \text{for all}\ v\in\sn \}$, $i\in[1,n]$, and $\tilde{C}_i(K,\cdot)$ are the dual curvature measures, which are dual to Federer's curvature measures and play the same roles in their Minkowski types problems as area measures and curvature measures. The Minkowski problems corresponding to the dual curvatures measure are called the dual Minkowski problem. In \cite{MR3573332}, the authors provided the sufficient condition for the existence of o-symmetric solutions. Subsequently, numerous studies have been conducted on the dual Minkowski problem, including references \cite{MR4008522,MR3953117,MR4271790,MR4055992}. 
	
Furthermore, there are various other types of geometric invariants, such as capacity, with Minkowski type problems first explored by David Jerison \cite{MR1395668}. In addition, Colesanti-Nystr\"om-Salani-Xiao-Yang-Zhang \cite{MR3406534} investigated Minkowski-type problems for $p$-capacity when $1<p<n$. The chord integral, proposed by Lutwak-Xi-Yang-Zhang \cite{MR1111}, addresses the chord Minkowski problem and provides the necessary and sufficient conditions for the existence of solutions. Xi-Yang-Zhang-Zhao further studied the $L_p$ chord Minkowski problem in\cite{MR4533537}, solving the case for $p>1$ and the symmetric case for $0<p<1$.
	
	The primary focus of this article is on the Gaussian volume functional in Gaussian probability space, initially proposed by Huang-Xi-Zhao in \cite{MR4252759}. The Gaussian volume $\gamma_n$ of a convex set $E$ in $\rn$ is defined by 
	\begin{equation*}\label{Gaussian probability measure}
		\gamma_n(E)=\frac{1}{(\sqrt{2\pi})^n}\int_Ee^{-\frac{|x|^2}{2}}dx,
	\end{equation*} 
	which is neither translation invariant nor homogeneous.  The surface area measure $S_{\gamma_n,K}(\cdot)$ in the Gaussian probability space is called the Gaussian surface area measure and defined on Borel $\eta\subset\sn$,
	\[
	S_{\gamma_n,K}(\eta)=\frac{1}{(\sqrt{2\pi})^n}\int_{v^{-1}_K(\eta)}e^{-\frac{|x|^2}{2}}d\mathcal{H}^{n-1}(x).
	\] 
	The variational formula presented in \cite{MR4252759} essentially applies the variational formula from \cite{MR3573332}, making it straightforward to obtain the derivative of Gaussian volume $\gamma_n$ on the set of convex bodies,
	\begin{equation}\label{variational of gaussian volume}
		\lim\limits_{t\rightarrow 0}\frac{\gamma_n(K+tL)-\gamma_n(K)}{t}=\int_{\sn}h_LdS_{\gamma_n, K}
	\end{equation}
	for any convex bodies $K$ and $L$ containing the origin in their interiors.
	
	\textbf{The Gaussian Minkowski problem.} Given a finite Borel measure $\mu$, what are the necessary and sufficient conditions on $\mu$ so that there exists a convex body $K\in\mathcal{K}^n_o$ such that
	\[
	\mu=S_{\gamma_n, K}?
	\] 
	
	In \cite{MR4252759}, the authors used the degree theory to establish the existence of solutions for the Gaussian Minkowski problem. Specifically, they demonstrated that for a 
	given even measure $\mu$ on $\sn$ that is not concentrated in any subspace and $|\mu|<\frac{1}{\sqrt{2\pi}}$, there exists a unique o-symmetric convex body $K$ with $\gamma_n(K)>\frac{1}{2}$ such that	$\mu=S_{\gamma_n,K}$.
	
	If $K$ is sufficiently smooth, its Gaussian surface area measure is absolutely continuous with respect to spherical Lebesgue measure. In this case, the Gaussian Minkowski problem reduces to solving the following Monge-Amp\`ere type equation on $\sn$, 
	\begin{equation}\label{monge ampere equation}
		\frac{1}{(\sqrt{2\pi})^n}e^{-\frac{|\nabla h|^2+h^2}{2}}\det(\nabla^2h+hI)=f, 
	\end{equation}
where $\nabla h$ and $\nabla^2 h$ are the gradient and the Hessian of $h$ on $\sn$ with respect to an orthonormal basis. It is obviously of interest to study the $L_p$-Gaussian Minkowski problem.

	\textbf{$L_p$-Gaussian Minkowski problem.} Given a finite Borel measure $\mu$, what are the necessary and sufficient conditions on $\mu$ so that there exists a convex body $K\in\mathcal{K}^n_o$ such that
	\[
	\mu=S_{p,\gamma_n, K}?
	\]  
	
	 Liu \cite{MR4358235} provided the variational formula of Gaussian volume $\gamma_n$ for the $L_p$ Minkowski sum, which applies the variational formula \eqref{variational of gaussian volume},
	\[
	\lim\limits_{t\rightarrow 0}\frac{\gamma_n(K+_ptL)-\gamma_n(K)}{t}=\int_{\sn}h_L^pdS_{p,\gamma_n, K} 
	\] 
	for any convex bodies $K$ and $L$ containing the origin in their interiors, where the $L_p$-Gaussian surface area measure $S_{p,\gamma_n,K}$ is given by
	\[
	S_{p,\gamma_n,K}(\eta)=\frac{1}{(\sqrt{2\pi})^n}\int_{v^{-1}_K(\eta)}e^{-\frac{|x|^2}{2}}(x,\nu_K(x))^{1-p}d\mathcal{H}^{n-1}(x). 
	\] 
	
    The main results of \cite{MR4358235} state that
	for $p\geq1$, if $\mu$ is a nonzero finite even measure on $\sn$ that is not concentrated in any subspace and $|\mu|<\sqrt{\frac{2}{\pi}}r^{-p}ae^{-\frac{a^2}{2}}$, where $r$ and $a$ are chosen such that $\gamma_n(rB)=\gamma_n(P)=\frac{1}{2}$, symmetry strip $P=\{x\in \mathbb{R}^n: |x_1|\leq a\}$, then there exists a unique o-symmetric $K$ with $\gamma_n(K)>\frac{1}{2}$ such that
	\[
	\mu=S_{p,\gamma_n,K}. 
	\] 
	Its corresponding Monge-Amp\`ere type equation on $\sn$ is given by
	\begin{equation}\label{Lp Gaussian Minkowski Problem}
		\frac{1}{(\sqrt{2\pi})^n}h^{1-p}e^{-\frac{h^2+|\nabla h|^2}{2}}\det(\nabla^2h+hI)=f. 
	\end{equation}
	
	Recently, Feng-Hu-Xu utilized the continuous method to derive the existence of smooth solutions for equation \eqref{Lp Gaussian Minkowski Problem} without volume limitation when $p>n$ in \cite{MR4564937}. They also derived an asymmetric smooth solution for equation \eqref{Lp Gaussian Minkowski Problem} with $\gamma_n(K)>\frac{1}{2}$ when $1\leq p<n$. 
	
	\textbf{Existence of solutions with small volume.} It is evident to highlight a crucial restrictive condition, where $\gamma_n(K)>\frac{1}{2}$, for the existence and uniqueness of solutions for equation \eqref{Lp Gaussian Minkowski Problem} in both aforementioned works.  The condition is intrinsically determined by the Ehrhard inequality, which can be utilized to prove the uniqueness of solutions when Gaussian volume exceeds $\frac{1}{2}$.  Chen-Hu-Liu-Zhao \cite{MR4658824} raised the question of whether there exists a solution $K$ to the $L_p$-Gaussian Minkowski problem with $\gamma_n(K)<\frac{1}{2}$? They demonstrated in $\mathbb{R}^2$ space that such even solutions do indeed exist.  The end of the article also points out that if the uniqueness of equation \eqref{Lp Gaussian Minkowski Problem} holds for $f\equiv c_0$, then there also exist high-dimensional solutions with Gaussian volume less than $\frac{1}{2}$ for equation \eqref{Lp Gaussian Minkowski Problem}. Fortunately, Ivaki-Milman provided this uniqueness in Corollary 1.4 in \cite{ivaki2023uniqueness}, and the existence of high-dimensional even solutions with small volume can also be obtained.
	
	In this article, we mainly study the case when $1\leq p<n$ with Gaussian volume less than $\frac{1}{2}$. Applying the uniqueness established in Theorem 1.3 of  \cite{ivaki2023uniqueness} and degree theory, we can extend the analysis to the $L_p$-Gaussian Minkowski problem in the small Gaussian volume setting and derive an even smooth solution to the Monge-Amp\`ere equation \eqref{Lp Gaussian Minkowski Problem}. 
	
	\begin{theo}\label{Existence of smooth,  small solution}
		Let $\alpha\in(0, 1)$, $1\leq p<n$, and $f\in C^{2, \alpha}_+(S^{n-1})$ be a positive even function and satisfy $\|f\|_{L_1}<(\frac{1}{\sqrt{2\pi}})^p(\frac{n}{2})^{1-p}$. Then there exists a $C^{4, \alpha}$ o-symmetric convex body $K$ with $\gamma_n(K)<\frac{1}{2}$ , its support function $h$ such that
		\begin{equation*}
			\frac{1}{(\sqrt{2\pi})^n}e^{-\frac{|\nabla h|^2+h^2}{2}}h^{1-p}\det(h_{ij}+hI)=f. 
		\end{equation*}
	\end{theo}
	
	Through an approximation argument and combining the conclusion in \cite{MR4564937} with the existence of a solution with Gaussian volume exceeding $\frac{1}{2}$, there exist at least two solutions to the $L_p$-Gaussian Minkowski problem when $1\leq p<n$.
	\begin{theo}\label{approximation}
		For $1\leq p<n$, if $\mu$ is an even Borel measure and absolutely continuous with respect to the spherical Lebesgue measure $\nu$, i.e., $d\mu=fd\nu$, and $\frac{1}{C}<f<C$ for some positive constant $C$ and $\|f\|_{L^1}<(\frac{1}{\sqrt{2\pi}})^p(\frac{n}{2})^{1-p}$, then there exist at least two o-symmetric convex bodies $K_1$, $K_2$ such that
		\[
		S_{p, \gamma_n, K_1}=S_{p, \gamma_n, K_2}=\mu. 
		\]
	\end{theo}
	
	\section{Preliminaries}
	Our notations and some properties are based on Schneider\cite{MR3155183}.
	
	 Let $\rn$ be the $n$-dimensional Euclidean space.  The unit sphere in $\rn$
	is denoted by $\sn$. A convex body in $\rn$ is a compact convex set with nonempty interior. 
	Let $K$ be a convex body in $\rn$ that contains the origin in its interior. 
	The radial function $\rho_K$ is defined by
	\[
	\rho_K(x) = \max\{\lambda : \lambda x \in K\},  \quad x\in \rn\setminus \{0\}. 
	\]
	For $u\in \sn$, it is clear that $\rho_K(u)u \in \partial K$.  The support function
	$h_K$ of $K$ is defined by
	\[
	h_K(x) = \max\{x\cdot y : y\in K\},  \quad x\in\rn. 
	\]
	
	The radial function and the support function are related by the equations, 
	\begin{align*}
		h_K(v) &= \sup_{u\in\sn} \{\rho_K(u) u\cdot v\},  \\
		\frac1{\rho_K(u)} &= \sup_{v\in\sn} \frac{u\cdot v}{h_K(v)}. 
	\end{align*}
	
	Denote by $\mathcal{K}^n$ the class of compact convex sets in $\rn$, and denote by $\mathcal K^n_o$ the class of compact convex sets that contain the origin in their interior in $\rn$.  Also denote by $\mathcal K^n_e$ the class of o-symmetric convex bodies in $\rn$. The Hausdorff metric of the sets $K,L\in\mathcal{K}^n$ is defined by 
	\[
	d(K,L)=\min\{t\geq0:K\subset L+tB, L\subset K+tB\}, 
	\]
and it is equivalent to 
	\[
	d(K,L)= \max_{u\in \sn}|h(K,u)-h(L,u)|, 
	\]
	see subsection 1.8 in Schneider\cite{MR3155183}. We say a sequence $\{K_i\}$ of compact convex sets is  convergent to $K$ in Hausdorff metric, if it satisfies
	\[
	d(K_i,K)\rightarrow0
	\] 
	as $i\rightarrow\infty$. 
	
	For $x\notin K$,  denote by $d(K, x)$ the distance from $x$ to $K$.  There exists
	a unique point $p_K(x) \in \partial K$ so that
	\[
	d(K, x) = |x-p_K(x)|. 
	\]
	Denote by $v_K(x)$ the outer normal unit vector of $\partial K$ at $p_K(x)$
	defined by
	\[
	v_K(x) = \frac{x-p_K(x)}{d(K, x)},
	\]
	and
	\[
	\alpha_K(x)=v_K(\rho_K(x)x). 
	\]
		
	\textbf{$L_p$-Gaussian surface area measure.}
	For $p\in \mathbb{R}$, the $L_p$-Gaussian surface area measure $S_{p,\gamma_n,K}(\cdot)$ on the unit sphere $\sn$ is defined as 
	\[
	S_{p,\gamma_n,K}(\eta)=\frac{1}{(\sqrt{2\pi})^n}\int_{v^{-1}_K(\eta)}e^{-\frac{|x|^2}{2}}(x,\nu_K(x))^{1-p}d\mathcal{H}^{n-1}(x),
	\]
	where $\eta\subset \sn$ is a Borel set. 
	
	The $L_p$-Gaussian surface area measure converges weakly with respect to Hausdorff metric. 
	
	\begin{theo}\label{weak convergent}
		Let $K_i\in \mathcal{K}^n_o$ such that $K_i$ converges to $K_0\in\mathcal{K}^n_o$ in Hausdorff metric.  Then $S_{p,\gamma_n,K_i}$ converges to $S_{p,\gamma_n,K_0}$ weakly. 
	\end{theo}
	\begin{proof}
		Since $K_0\in\mathcal{K}^n_o$, there exists a uniform constant $C>0$ such that 
		\[
		\frac{1}{C}B\subset K_i\subset CB
		\] 
		for sufficiently large $i$.  Then we have
		\begin{equation}\label{3. 1}
			h_{K_i}^{1-p}e^{-\frac{\rho_{K_i}(u)^2}{2}}\rho_{K_i}^{n-1}(u)\rightarrow
			h_{K_0}^{1-p}e^{-\frac{\rho_{K_0}(u)^2}{2}}\rho_{K_0}^{n-1}(u)
		\end{equation}
		as $i\rightarrow \infty$. On the other hand, we consider a function $g\in C(\sn)$.  Since $\alpha_{K_i}\rightarrow\alpha_{K_0}$ a.e. with respect to spherical Lebesgue measure,
		\begin{equation}\label{3. 2}
			\frac{g(\alpha_{K_i}(u))}{u\cdot\alpha_{K_i}(u)}\rightarrow
			\frac{g(\alpha_{K_0}(u))}{u\cdot\alpha_{K_0}(u)}
		\end{equation}
		as $i\rightarrow \infty$.  Combining \eqref{3. 1} and \eqref{3. 2}, we have
		\begin{align*}
			\int_{\sn}gdS_{p,\gamma_n,K_i}&=\frac{1}{(\sqrt{2\pi})^n}\int_{\partial K_i}g(\nu_{K_i}(x))(x\cdot\nu_{K_i}(x))^{1-p}e^{-\frac{|x|^2}{2}}d\mathcal{H}^{n-1}(x)\\
			&=\frac{1}{(\sqrt{2\pi})^n}\int_{\sn}g(\alpha_i(u))h_{K_i}(\alpha_{K_i}(u))^{1-p}e^{-\frac{\rho_{K_i}(u)^2}{2}}\frac{\rho_{K_i}(u)^n}{h_{K_i}}(\alpha_{K_i}(u))du\\
			&=\frac{1}{(\sqrt{2\pi})^n}\int_{\sn}\frac{g(\alpha_i(u))}{u\cdot\alpha_{K_i}(u)}h_{K_i}(\alpha_{K_i}(u))^{1-p}e^{-\frac{\rho_{K_i}(u)^2}{2}}\rho_{K_i}(u)^{n-1}du\\
			&\rightarrow\frac{1}{(\sqrt{2\pi})^n}\int_{\sn}\frac{g(\alpha_0(u))}{u\cdot\alpha_{K_0}(u)}h_{K_0}(\alpha_{K_0}(u))^{1-p}e^{-\frac{\rho_{K_0}(u)^2}{2}}\rho_{K_0}(u)^{n-1}du\\
			&=\int_{\sn}gdS_{p,\gamma_n,K_0}. 
		\end{align*}
	\end{proof}

	\section{a priori estimate}
This section provides some a priori estimates. We will then utilize these estimates in degree theory to establish the existence of smooth solutions to the $L_p$-Gaussian Minkowski problem for $1\leq p<n$. 
	
	The following $C^0$ estimate is motivated by \cite{MR4658824}.
	
	\begin{lemm}\label{C0 estimate}\textbf{($C^0$ estimate)}
		Suppose $0<p<n$,  f is an even positive function on $\sn$ and $h\in C^2_+(\sn)$ is an even solution to equation \eqref{Lp Gaussian Minkowski Problem}. If there exists $\tau>0$ such that
		\begin{equation}
			\frac{1}{\tau}<f<\tau, 
		\end{equation}
		then there exists $\tau'>0$ (which depends only on $\tau$) such that
		\begin{equation}
			\frac{1}{\tau '}<h<\tau '. 
		\end{equation}
		\begin{proof}
			Firstly, we show that $h$ is bounded from above. Assume $h$ achieves its maximum at $v_0\in \sn$, i.e., $h(v_0)=h_{max}$. Then by $\eqref{Lp Gaussian Minkowski Problem}$, we have
			\[
			\frac{1}{(\sqrt{2\pi})^n}h_{max}^{1-p}e^{-\frac{h_{max}^2}{2}}h_{max}^{n-1}\geq f(v_0)>\frac{1}{\tau}. 
			\]
			Note that $\frac{1}{(\sqrt{2\pi})^n}t^{n-p}e^{-\frac{t^2}{2}}\rightarrow 0$ as $t\rightarrow\infty$, then there exists a $\tau_1>0$ such that
			\[
			h_{max}<\tau_1. 
			\]
			Next, we show that $h_{max}$ is bounded from below. By $\eqref{Lp Gaussian Minkowski Problem}$, we conclude
			\[
			h_{max}^{n-1}\geq \det(h_{ij}+h\delta_{ij})|_{v_0}=f(v_0)(\sqrt{2\pi})^nh^{p-1}(v_0)e^{\frac{h^2(v_0)}{2}}\geq C(\inf f)h_{max}^{p-1}, 
			\]
			which implies that there exists a constant $\tau_2$ such that
			\begin{equation}\label{hmax bound}
				h_{max}\geq \tau_2. 
			\end{equation}
			Finally, we show that $h$ is also bounded from below. Assume that $h$ achieves its minimum at $u_0\in S^{ n-1}$ and $h(u_0)=h_{min}$. Then by $\eqref{Lp Gaussian Minkowski Problem}$, we have
			\[
			\frac{1}{n}hdet(h_{ij}+h\delta_{ij})=\frac{(\sqrt{2\pi})^n}{n}fh^pe^{\frac{h^2+|\nabla h|^2}{2}}\geq Ch^p. 
			\]
			Observe that the total integral of $\frac{1}{n}hdet(h_{ij}+h\delta_{ij})$ on $\sn$ is the volume of $K$. Thus
			\begin{equation}\label{total integral}
				\mathcal{H}^n(K)\geq C\int_{\sn}h^p(v)dv. 
			\end{equation}
			Since $h$ is an even solution, we derive
			\begin{equation}\label{estimate of h by hmax}
				h(v)\geq h_{\max}|v\cdot v_0|. 
			\end{equation}
			Combining $\eqref{hmax bound}$, $\eqref{total integral}$ and $\eqref{estimate of h by hmax}$, we have
			\[
			\mathcal{H}^n(K)\geq Ch_{\max}^p\int_{\sn}|v\cdot v_0|^pdv=C'. 
			\]
			On the other hand, since
			\[
			K\subset (h_{\max}B_1)\cap\{x\in \mathbb{R}^{n}:|x\cdot u_0|\leq h_{\min}\}, 
			\]
			thus
			\[
			\mathcal{H}^n(K)\leq 2^nh_{\max}^{n-1}h_{\min}<C(\tau_1)h_{\min}. 
			\]
			As a consequence, there exists a $\tau_3$ such that:
			\[
			h_{\min}>\tau_3>0. 
			\]
		\end{proof}
	\end{lemm}
	
	The following higher order estimations utilize the method in \cite{MR4252759}.
	
	\begin{lemm}\label{a-priori estimate} For $0<\alpha<1$ and $0<p<n$,  assume that the even function $f\in C^{2, \alpha}(\sn)$ such that $\frac{1}{C}<f<C$ for some positive constants $C$. For $K\in\mathcal{K}^n_e$, if its support function $h\in C^{4, \alpha}(\sn)$ satisfies
		\begin{equation}\label{lp monge ampere equation}
			\frac{1}{(\sqrt{2\pi})^n}h^{1-p}e^{-\frac{h^2+|\nabla h|^2}{2}}\det(h_{ij}+h\delta_{ij})=f, 
		\end{equation} 
		then there exists a positive $C'$ (which only depends on $C$) such that\\
		(1) $C^1$ estimate: $\frac{1}{C'}<\sqrt{|\nabla h|^2+h^2}<C'.$\\
		(2) $C^2$ estimate: $\frac{1}{C'}I<(\nabla^2h+hI)<C'I$.\\
		(3) $|h|_{C^{4, \alpha}}<C'$.

		\begin{proof}
			(1) Note that by Lemma \ref{C0 estimate}, there exists a $C'>0$ such that $\frac{1}{C'}<h<C'$, which implies $\frac{1}{C'}B\subset K \subset C'B$. By the definition of support function,  
			\[
			\nabla h(v)+h(v)v=\nu^{-1}_K(v)\in \partial K
			\]	
			for each $v\in \sn$. Thus the lower and upper bounds on $|\nabla h|^2+h^2$ comes from the bounds on $K$. 
			
			(2) To get the $C^2$ estimate, we divide it into two claims.\\
			\textbf{Claim 1}: The trace of the matrix $(\nabla^2h+hI)$ is bounded from above. \\
			\textbf{Claim 2}: The determinant of the matrix $(\nabla^2h+hI)$ is bounded from above and below by a positive constant. 
			
			By combining \textbf{Claim 1} and \textbf{Claim 2} together, we can immediately conclude that all eigenvalues of $(\nabla^2h+hI)$ have positive upper and lower bounds, which means the matrix $(\nabla^2h+hI)$ is positive definite. 
			
			Firstly, we can easily prove \textbf{Claim 1} by using the equation \eqref{lp monge ampere equation}. Indeed, 
			\[
			\det(\nabla^2h+hI)=(\sqrt{2\pi})^nh^{p-1}fe^{\frac{|\nabla h|^2+h^2}{2}}, 
			\] 
			where the right side has positive upper and lower bounds based on the bounds of $f$,  $h$ and (1). 
			
			To prove \textbf{Claim 2},  set
			\begin{equation}\label{commutator identity}
				H=tr(\nabla^2h+hI)=\Delta h+(n-1)h. 
			\end{equation}
			Due to the continuity of $H$ on $\sn$,  there exists a $v_0\in \sn$ such that $H(v_0)=\max_{v\in \sn}H$.  Then $\nabla H(v_0)=0$ and the matrix $\nabla^2H(v_0)$ is negative semi-definite. Choose a local orthonormal frame $e_1, \cdots, e_{n-1}$  such that $(h_{ij})$ is diagonal. Since the commutator identity
			\[
			H_{ii}=\Delta\omega_{ii}-(n-1)\omega_{ii}+H, 
			\] 
			where $\omega^{ij}$ denote the inverse of the matrix $\omega_{ij}=((h_{ij})+h\delta_{ij})$, the facts that $(\omega^{ij})$ is positive definite, $\nabla^2H$ is negative semi-definite,  and $(\omega^{ij})$ is diagonal at $v_0$ imply
			\begin{equation}\label{main inequality}
				0\geq \omega^{ii}H_{ii}=\omega^{ii}\Delta\omega_{ii}+H\sum_i\omega^{ii}-(n-1)^2\geq\omega^{ii}\Delta\omega_{ii}-(n-1)^2. 
			\end{equation}
			Taking the logarithm of equation \eqref{lp monge ampere equation}, then
			\[
			\log \det(\nabla^2h+hI)=\log f+\frac{|\nabla h|^2+h^2}{2}+\frac{n}{2}\log2\pi+(p-1)\log h. 
			\]
			Taking the spherical Laplacian of the above equation, we derive
			\begin{align}\label{a}
				&\sum_{\alpha}(\omega^{ij})_{\alpha}(\omega_{ij})_{\alpha}+\omega^{ij}\Delta\omega_{ij}\notag \\
				=&\Delta \log f+\sum_{i, j}h^2_{ij}+\sum_ih_i\Delta h_i+(1-\frac{p-1}{h^2})|\nabla h|^2+(h+\frac{p-1}{h})\Delta h. 
			\end{align}
			By the definition of $H$ and third
			covariant derivative of $h$ in \cite{Cheng-Yau}, we have
			\begin{align}\label{taking spherical laplacian}
			  \sum_ih_i\Delta h_i&=\sum_ih_i\sum_jh_{ijj}\notag\\
			  &=\sum_ih_i[\sum_j(h_{jji}-h_j\delta_{ij}+h_i)]\notag\\
			  &=\sum_ih_i[\sum_jh_{jji}+(n-1)h_i-h_i]\notag\\
			  &=\nabla h\cdot \nabla H-|\nabla h|^2
			\end{align}
		and 
		\begin{equation}\label{hij^2}
			\sum_{i,j}h^2_{ij}=H^2-2hH+(n-1)h^2.
		\end{equation}
			We also note that
			\begin{equation}\label{h laplace h}
				(h+\frac{p-1}{h})\Delta h=(h+\frac{p-1}{h})H-(n-1)h^2-(p-1)(n-1). 
			\end{equation}
	        Then, by using the fact that $(\omega^{ij})$ is the inverse matrix of $(\omega_{ij})$, we derive
			\[
			(\omega^{ij})_\alpha(\omega_{jk})\alpha=-\omega^{im}(\omega_{ml})_{\alpha}\omega^{lj}(\omega_{jk})_\alpha, 
			\]
			which implies its trace is non-positive, i.e.,
			\begin{equation}\label{b}
				(\omega^{ij})_\alpha(\omega_{ij})_\alpha\leq0. 
			\end{equation}
			Combining \eqref{a}, \eqref{taking spherical laplacian}, \eqref{hij^2}, \eqref{h laplace h} and \eqref{b}, we conclude
			\[
			\omega^{ij}\Delta\omega_{ij}\geq\Delta \log f+H^2+(\frac{p-1}{h}-h)H+\nabla h\cdot\nabla H-\frac{p-1}{h^2}|\nabla h|^2-(p-1)(n-1). 
			\]
			When evaluated at $v_0$ where $H$ reaches its maximum,  we have 
			\begin{equation}\label{wii inequality}
				\omega^{ii}\Delta\omega_{ii}\geq H^2+(\frac{p-1}{h}-h)H+[\Delta \log f-\frac{p-1}{h^2}|\nabla h|^2-(p-1)(n-1)]. 
			\end{equation}
			Then \eqref{main inequality} and \eqref{wii inequality} imply
			\[
			0\geq H^2+(\frac{p-1}{h}-h)H+[\Delta\log f-\frac{p-1}{h^2}|\nabla h|^2-(p+n-2)(n-1)]. 
			\]
			Note that the right side is a quadratic polynomial to $H$. Then by the bounds on $f$,  $|f|_{C^{2, \alpha}}$,  $h$ and $|\nabla h|^2$,  the coefficients have bounds which only depend on $C$  and we can deduce that $H$ is bounded from above by a positive constant $C'$ which only depends on $C$. 
			
			(3) By using the standard Evans-Krylov-Safonov theory, higher order estimates $|h|_{C^{4, \alpha}}$ can be derived, where we use statement (2) to ensure that the equation \eqref{lp monge ampere equation} is uniformly elliptic. 	
		\end{proof}
	\end{lemm}

	\section{Existence of small solutions}
	Before proving the existence of solutions for the $L_p$-Gaussian Minkowski problem with Guassian volume less than $\frac{1}{2}$, we still require some lemmas. 
	The first lemma is from Feng-Hu-Xu\cite{MR4564937}, which provides the $L_p$-Gaussian isoperimetric inequality. 
	\begin{lemm}\label{Lp Gaussian isoperimetric inequality}\textbf{($L_p$-Gaussian isoperimetric inequality)}
		Let $K\in \mathcal{K}^n_o$.  Then for $p\geq 1$, 
		\[
		|S_{p, \gamma_n, K}|\geq n\gamma_n(K)(\frac{\varphi(\Gamma^{-1}(\gamma_n(K)))}{n\gamma_n(K)})^p, 
		\]
		where $\varphi(t)=\frac{e^{-\frac{t^2}{2}}}{\sqrt{2\pi}}$,  and $\Gamma(x)=\frac{1}{\sqrt{2\pi}}\int_{-\infty}^{x}e^{-\frac{t^2}{2}}dt$. Especially,  if $K\in\mathcal{K}^n_o$ such that $\gamma_n(K)=\frac{1}{2}$, then for $p\geq1$,  we have
		\[
		|S_{p, \gamma_n, K}|\geq (\frac{1}{\sqrt{2\pi}})^p(\frac{n}{2})^{1-p}. 
		\]
	\end{lemm}
	
	The next lemma concerns the uniqueness of smooth solutions to the isotropic Gaussian Minkowski problem, first obtained by Ivaki-Milman \cite{ivaki2023uniqueness}. This lemma constitutes a key element in employing degree theory to establish the existence of solutions. 
	\begin{lemm}\label{uniqueness of isotropic equation}
		For $p\geq1$, if there exists a convex body  $K\in\mathcal{K}^n_e$ satisfies that its support function $h$ is the solution of the following equation
		\begin{equation}\label{isotropic equation}
			h^{1-p}e^{-\frac{|\nabla h|^2+h^2}{2}}\det(\nabla^2h+hI)=C, 
		\end{equation} 
		where $C>0$ is a constant, then $h$ has to be a constant solution. 
	\end{lemm} 
	\begin{proof}
		We apply the Theorem 1.3 in \cite{ivaki2023uniqueness}, and set $\varphi=ch^{p-1}e^{\frac{h^2+|\nabla h|^2}{2}}$. Then for $p\geq1$,  if $h$ is the solution of equation \eqref{isotropic equation},  it must be a constant.    
	\end{proof}
	
	By utilizing Lemma \ref{uniqueness of isotropic equation}, we can derive the following conclusion through straightforward calculations.
	
	\begin{lemm}\label{number of solution}
		For $1\leq p<n$,  simplify equation \eqref{isotropic equation} to the following 
		\begin{equation}\label{simplify}
			h^{1-p}e^{-\frac{|\nabla h|^2+h^2}{2}}\det(\nabla^2h+hI)=C, 
		\end{equation}
		and we have\\
		(a) if $C\in(0, e^{-\frac{n-p}{2}}(n-p)^{\frac{n-p}{2}})$,  there are precisely two constant solutions to \eqref{simplify};\\
		(b) if $C=e^{-\frac{n-p}{2}}(n-p)^{\frac{n-p}{2}}$,  there is a unique constant solution to \eqref{simplify};\\
		(c) if $C>e^{-\frac{n-p}{2}}(n-p)^{\frac{n-p}{2}}$,  there is no constant solution to \eqref{simplify}. 
	\end{lemm}
	
	Now we are ready to prove the existence result by using the degree theory. 
	
	\begin{theo}\label{Existence of smooth,  small solution b}
		 For $1\leq p<n$, let $\alpha\in(0, 1)$,   $f\in C^{2, \alpha}_+(\sn)$ is a positive even function and satisfies $\|f\|_{L_1}<(\frac{1}{\sqrt{2\pi}})^p(\frac{n}{2})^{1-p}$. Then there exists a $C^{4, \alpha}$ convex body $K\in \mathcal{K}^n_e$ with $\gamma_n(K)<\frac{1}{2}$ satisfies
		\begin{equation}\label{monge ampere equation for lp}
			\frac{1}{(\sqrt{2\pi})^n}e^{-\frac{|\nabla h|^2+h^2}{2}}h^{1-p}\det(h_{ij}+hI)=f. 
		\end{equation}
	\end{theo}
	
	\begin{proof}
		We prove the existence of solutions by using the degree theory for second-order nonlinear elliptic operators developed in Li\cite{MR1026774}. By employing Lemma \ref{number of solution}, it follows that there exists a unique constant solution $h\equiv r_0>0$ with $\gamma_n([h])<\frac{1}{2}$ if $f\equiv c_0>0$ is small enough.  We also require that $c_0>0$ is small enough so that $|c_0|_{L_1}<(\frac{1}{\sqrt{2\pi}})^p(\frac{n}{2})^{1-p}$.  Our final requirement for $c_0>0$ is that the operator $L\phi=\Delta_{\sn}\phi+((n-p)-r_0^2)\phi$ is invertible,  which is possible since the spherical Laplacian has a discrete spectrum. 
		
		Define $F(\cdot; t):C^{4, \alpha}(\sn)\rightarrow C^{2, \alpha}(\sn)$ as follows
		\[
		F(h; t)=\det(\nabla^2h+hI)-(\sqrt{2\pi})^ne^{\frac{|\nabla h|^2+h^2}{2}}h^{p-1}f_t, 
		\] 
		where 
		\[
		f_t=(1-t)c_0+tf. 
		\]
		Since $f\in C^{2, \alpha}(\sn)$ and  $c_0>0$,  there exists $C>0$ such that $\frac{1}{C}<f, c_0<C$ and $|f|_{C^{2, \alpha}}<C$. Note that for each $t\in[0, 1]$, the function $f_t$ has the same bound as $f$, that is
		\[
		\frac{1}{C}<f_t<C,\  |f_t|_{L_1}<(\frac{1}{\sqrt{2\pi}})^p(\frac{n}{2})^{1-p},\  |f_t|_{C^{2, \alpha}}<C. 
		\]
		Define $O\subset C^{4, \alpha}(\sn)$ by
		\[
		O=\{h\in C^{4, \alpha}(\sn):\ \frac{1}{C'}<h<C',\  \frac{1}{C'}I<(\nabla^2h+hI)<C'I,\  |h|_{C^{4, \alpha}}<C',\  \gamma_n(h)<\frac{1}{2}\}, 
		\]
		where $\gamma_n(h)=\gamma_n([h])$ is well-defined,  and $C'>0$ is the constant extracted from Lemma \ref{a-priori estimate}. We note that $O$ is an open bounded set under the norm $|\cdot|_{C^{4, \alpha}}$. 
		
    	Since every $h\in O$ has uniform upper and lower bounds for the eigenvalues of its Hessian,  the operator $F(\cdot, t)$ is uniformly elliptic on $O$ for any $t\in [0, 1]$. We claim that for each $r\in [0, 1]$,  if $h\in\partial O$,  then
		\[
		F(h; t)\neq0. 
		\]
		If not, i.e., $ F(h; t)=0$,  then $h$ solves
		\begin{equation}\label{t Guassian equation}
			\frac{1}{(\sqrt{2\pi})^n}e^{-\frac{|\nabla h|^2+h^2}{2}}h^{1-p}\det(\nabla^2h+hI)=f_t. 
		\end{equation} 
		Since $h\in\partial O$,  we also have that $\gamma_n(h)\leq\frac{1}{2}$.  If $\gamma_n(h)<\frac{1}{2}$,  then by Lemma \ref{C0 estimate} and Lemma  \ref{a-priori estimate},  we have
		\[
		\frac{1}{C'}<h<C',\  \frac{1}{C'}I<(\nabla^2h+hI)<C'I,\  |h|_{C^{4, \alpha}}<C'. 
		\]
		This is a contradiction to the openness of $O$.  Thus,  for any $h\in\partial O$,  it should satisfy
		\[
		\gamma_n(h)=\frac{1}{2}. 
		\]
		Then by Lemma \ref{Lp Gaussian isoperimetric inequality},  there holds
		\[
		|S_{p,\gamma_n, [h]}|\geq(\frac{1}{\sqrt{2\pi}})^p(\frac{n}{2})^{1-p}, 
		\]
		but this contradicts the fact that $h$ solves equation \eqref{t Guassian equation} and the condition $|f_t|_{L_1}<(\frac{1}{\sqrt{2\pi}})^p(\frac{n}{2})^{1-p}$. Therefore, we prove the claim. 
		
		Thus we use Proposition 2.2 in Li\cite{MR1026774} and conclude that
		\[
		\deg(F(\cdot; 0), O, 0)=\deg(F(\cdot;1), O, 0). 
		\]
		
		Next we compute $\deg(F(\cdot;0), O, 0)$.  For simplicity,  we will write $F(\cdot)=F(\cdot;0)$. 
		Recall that $h\equiv r_0$ is the unique solution in $O$ to equation \eqref{monge ampere equation for lp} when $f\equiv c_0$.  Denote by $L_{r_0}:C^{4, \alpha}(\sn)\rightarrow C^{2, \alpha}(\sn)$ the linearized operator of $F$ at the constant function $r_0$.  It is straightforward to compute that 
		\begin{align*}
			L_{r_0}(\phi)
			&=r_0^{n-2}\Delta_{\sn}\phi+((n-p)r_0^{n-2}-r_0^n)\phi    \\	&=r_0^{n-2}(\Delta_{\sn}\phi+((n-p)-r_0^2)\phi), 
		\end{align*}
		which we have specifically chosen a $c_0>0$ such that $L_{r_0}$ is invertible.  By Proposition 2.3 in Li\cite{MR1026774} and the fact that $h\equiv r_0$ is the unique solution for $F(h)=0$ in $O$,  we have
		\[
		\deg(F, O, 0)=\deg(L_{r_0}, O, 0)\neq0, 
		\] 
		where the last inequality follows from the Proposition 2.4 in Li\cite{MR1026774}.  Then we conclude that
		\[
		\deg(F(\cdot;1), O, 0)\neq0, 
		\]
		which implies the existence of $h\in O$ such that $F(h;1)=0$.

	\end{proof}
	
	We can immediately remove the regularity assumption on $f$ through a simple approximation. 
	
	\begin{theo}\label{approximation for small volume}
		Let $f\in L^1(\sn)$ be an even function such that $\|f\|_{L_1}<(\frac{1}{\sqrt{2\pi}})^p(\frac{n}{2})^{1-p}$.  If there exists $C>0$ such that $\frac{1}{C}<f<C$ on $\sn$,  then there exists an o-symmetric $K$ with $\gamma_n(K)<\frac{1}{2}$ such that
		\[
		dS_{p, \gamma_n, K}=fdv. 
		\]
	\end{theo}
	
	\begin{proof}
		Suppose that $d\mu_i=f_idv$ converges weakly to $fdv$, where $f_i\in C^{2,\alpha}_e(\sn)$ and $\frac{1}{C}<f_i<C$ on $\sn$. Since $\|f\|_{L_1}<(\frac{1}{\sqrt{2\pi}})^p(\frac{n}{2})^{1-p}$,  $\mu_i$ satisfies $|\mu_i|<(\frac{1}{\sqrt{2\pi}})^p(\frac{n}{2})^{1-p}$ for large enough $i$. Then by Theorem \ref{Existence of smooth,  small solution b}, for every $f_i$, there exists a $C^{4, \alpha}$ o-symmetric $K_i$ with $\gamma_n(K_i)<\frac{1}{2}$, its support function $h$ such that
		\[
		\frac{1}{(\sqrt{2\pi})^n}e^{-\frac{|\nabla h|^2+h^2}{2}}h^{1-p}\det(h_{ij}+hI)=f_i. 
		\] 
		By Lemma \ref{C0 estimate},  there exists $C'>0$ that only depends on $C$ such that
		\begin{equation}\label{C0 estimate of convex}
			\frac{1}{C'}B\subset K_i\subset C'B. 
		\end{equation}
		Then we apply Blaschke's selection theorem and select a subsequence of $K_i$ converges to an o-symmetric convex body $K\in\mathcal{K}^n_e$. By \eqref{C0 estimate of convex},  we have
		\[
		\frac{1}{C'}B\subset K\subset C'B. 
		\]
		Then Theorem \ref{weak convergent} implies
		\[
		dS_{p, \gamma_n, K}=fd\mu. 
		\]
		Moreover, by the continuity of $\gamma_n$,  we have $\gamma_n(K)\leq\frac{1}{2}$. We can further deduce that $\gamma_n(K)<\frac{1}{2}$. If not, then by the inequality  $\|f\|_{L_1}<(\frac{1}{\sqrt{2\pi}})^p(\frac{n}{2})^{1-p}$  and the $L_p$-Gaussian isoperimetric inequality in Lemma \ref{Lp Gaussian isoperimetric inequality}, we obtain a contradiction. This completes the proof.  
	    \end{proof}
        Finally, combining the conclusion in \cite{MR4564937} with the existence of solutions with $\gamma_n>\frac{1}{2}$, the result immediately becomes to Theorem \ref{approximation}.
        
        \begin{proof}[Proof of Theorem \ref{approximation}]
        
        For every $f_i\in C^{2,\alpha}_e(\sn)$ and $\frac{1}{C}<f_i<C$,  Theorem 5.2 in \cite{MR4564937} provided that there exists an even solution $h_i\in C^{4,\alpha}$ with $\gamma_n({[h_i]})>\frac{1}{2}$.
      Note that the approximation in Theorem \ref{approximation for small volume} is independent of the size of Guassian volume, thus we can apporximate it similarly. 
        \end{proof}
	
	\section*{Acknowledgement}
	Deeply thanks to my supervisor, professor Yong Huang, for his meticulous help and encouragement. I am also thanks the referees for detailed reading and helpful comments.

\end{document}